\documentclass{amsart}

\usepackage{amssymb}
\usepackage{graphicx}
\usepackage{tikz}
\usepackage{enumitem}
\usepackage{physics}
\usepackage{hyperref}

\newcommand{\RR}{\mathbb{R}}

\newcommand{\KK}{\mathbb{K}}
\newcommand{\BB}{\mathcal{B}}

\newcommand{\C}{\mathrm{\check{C}}}
\newcommand{\VR}{\mathrm{VR}}

\newcommand{\Dirsum}{\bigoplus}

\newcommand{\compose}{\circ}

\DeclareMathOperator{\Rep}{Rep}
\DeclareMathOperator{\Hom}{Hom}

\DeclareMathOperator{\codim}{codim}
\DeclareMathOperator{\coker}{coker}
\DeclareMathOperator{\Qcodim}{Qcodim}
\DeclareMathOperator{\PH}{PH}

\DeclareRobustCommand{\stodd}{\text{\reflectbox{$\ddots$}}}

\newtheorem{introthm}{Main Theorem}
\newtheorem{thm}{Theorem}[section]

\newtheorem{lem}[thm]{Lemma}

\theoremstyle{definition}
\newtheorem{ex}[thm]{Example}
\newtheorem{rem}[thm]{Remark}
\newtheorem{defn}[thm]{Definition}
\newtheorem{question}{Question}

\title{Stabilization of codimension of persistence barcodes}
\author[J. Allman]{Justin Allman}
    \address{Department of Mathematics, Wake Forest University, Winston-Salem, NC, USA}
    \email{allmanjm@wfu.edu}
\author[A. Huang]{Anran Huang}
    \address{Department of Mathematics, University of Michigan, Ann Arbor, MI, USA}
    \email{huana@umich.edu}


\begin{document}

\begin{abstract}
Given a pointwise finite-dimensional persistence module over a totally ordered set $S$, a theorem of Crawley-Boevey guarantees the existence of a barcode. When the set $S$ is finite, the persistence module is an equioriented type-A quiver representation and the barcode identifies a distinguished point in an algebraic variety. We prove a formula for the codimension of this variety inside an ambient space of comparable representations which depends only on the combinatorics of the barcode. We therefore extend the notion of codimension to persistence modules over any set $S$ and prove that this extension is well-defined and can be effectively computed via a stabilization property of approximating quiver representations. Further, we prove realization theorems by constructing explicit examples via persistent homology of data for which the codimension can realize any natural number as its value.
\end{abstract}

\maketitle

\section{Introduction}
    \label{s:intro}

Fix a field $\KK$ and a partially ordered set $S$. A \emph{persistence module over $S$}, or an \emph{$S$-persistence module}, consists of
	\begin{enumerate}[leftmargin = *, label=(\roman*)]
	\item an indexed collection of $\KK$-vector spaces $\left(V_s ~|~ s\in S \right)$
	\item together with a collection of $\KK$-linear mappings $\left(f_{st}: V_s \to V_t~|~ s\leq t \right)$
	\end{enumerate}
where $f_{ss}$ is the identity map on $V_s$ and where we require the composition law $f_{tu}\compose f_{st} = f_{su}$ whenever $s,t,u\in S$ with $s \leq t \leq u$. When $V$ is an $S$-persistence module and $S' \subseteq S$, we obtain an $S'$-persistence module called the \emph{restriction} of $V$ to $S'$, denoted $V|S'$ by considering only those spaces and maps with indices in $S'$. 

Persistence modules are a central mathematical object of study in the burgeoning field of \emph{topological data analysis} (TDA) due to their appearance in computations for the \emph{persistent homology} (PH) of data; see Section \ref{s:PH} for examples. In this paper we introduce a new flavor of stability for persistence modules by considering a special class of restrictions which closely approximate a given persistence module. 

The traditional (and extremely useful!) notion of stability in persistence theory (PT) asserts that important ``signatures" (e.g., barcodes and persistence diagrams) exhibit a statistical robustness. However, when $S' \subseteq S$ is a finite set, algebro-geometric invariants from the representation theory (RT) of quivers can be computed. Morally, we imagine 
\begin{itemize}[leftmargin=*]
    \item \textbf{PT stability}. If two $S$-persistence modules only differ ``slightly'', then their persistence signatures only differ ``slightly''. 
    \item \textbf{RT stability}. For finite subsets $S'\subseteq S$, the value of a representation-theoretic invariant stabilizes on all ``nice enough'' restrictions $V|S'$.
\end{itemize}
We consider a specific representation theory / algebro-geometric invariant called \emph{codimension} for which we are able to give a precise formulation of an RT stability property. To our knowledge, the present work is the first example of any algebro-geometric invariant which admits RT stability.

\subsection{Persistence theory}
Throughout the paper we assume that our persistence modules are \emph{pointwise finite-dimensional}; i.e., all of the vector spaces $V_s$ are finite-dimensional $\KK$-spaces. Furthermore, we will restrict our attention to the case when $S$ is a \emph{totally} ordered set. In fact, in most of our examples $ S \subseteq \RR$. The quantities \[r_{st} = \rank(f_{st}:V_s \to V_t) \text{~for $s\leq t$}\] are called the \emph{persistent Betti numbers}, and 
the assumption that $\dim(V_s)<\infty$ for all $s\in S$ implies \[r_{st} \leq \min(\dim(V_s),\dim(V_t))<\infty \text{~for all $s\leq t$}.\] 

A subset $I \subseteq S$ is an \emph{interval} if $I \neq \emptyset$ and $s \leq t\leq u$ with $s,u\in I$ implies $t\in I$. Corresponding to an interval $I \subseteq S$, we define a distinguished $S$-persistence module, the so-called \emph{interval module} with $V_s = \KK$ whenever $s\in I$ and $V_u = 0$ whenever $u\notin I$. We set $f_{st} = 1$ for $s,t\in I$ with $s\leq t$ (and $f_{st} = 0$ otherwise). A theorem of Crawley-Boevey states 
\begin{thm}[\cite{CB}, Theorem 1.1]
	\label{thm:CB}
Every pointwise finite-dimensional persistence module over a totally ordered set $S$ is a direct sum of {interval modules}. \qed
\end{thm}
The collection of these intervals is called the \emph{barcode} of the persistence module, and because of this, in the sequel we will also refer to a single interval in a persistence module decomposition as a \emph{bar}. 
Throughout, for any bar $I \subseteq S$ we denote its \emph{birth} or \emph{birth time} $b(I) = \inf(I)$ and its \emph{death} or \emph{death time} by $d(I) = \sup(I)$. 
Whenever an interval $I$ satisfies $b(I)=d(I)=a$, we call $I$ a \emph{pointbar}. These occur whenever there exists a nonzero $v\in V_a$ which is simultaneously in $\coker(f_{sa})$ for all $s<a$ and $\ker(f_{at})$ for all $a<t$. 

The barcode provides a useful ``signature" of the corresponding persistence module due to its computability (in most applications) and its aforementioned robust PT-stability property. The most general result on PT-stability known to the authors is Theorem 5.23 and 5.24 of \cite{CdSGO}. 

\subsection{Representation theory}
When $S$ is a finite set, a pointwise finite-dimensional $S$-persistence module is an \emph{equioriented A-type quiver representation}, though in this paper we agree to simply say \emph{quiver representation}. In this setting, the result of Theorem \ref{thm:CB} is a special case of a theorem of Gabriel \cite{Gab}, and the \emph{barcode} is a representative ``normal form" for an entire isomorphism class (aka \emph{isoclass}) of quiver representations \cite{ADF1,ADF2,ADFK}. The isoclasses of quiver representations have an interesting history in their own right (well beyond A-type) and have numerous remarkable combinatorial and geometric properties; see, e.g., \cite{B-F,KMS,BR,Buc,KKR,Rim,All} and references therein.

In this paper, we focus on one such algebro-geometric invariant of the isoclass --- namely, its \emph{codimension} as a variety inside the affine \emph{ambient representation space} 
	\begin{equation}
		\label{eqn:rep}
		\Rep = \Dirsum_{i=1}^{n-1} \Hom\left(V_{s_i},V_{s_{i+1}}\right).
	\end{equation}

\subsection{Summary of results}
A formula for the codimension of an isoclass $\Omega\subset \Rep$ in terms of the persistent Betti numbers has been known since the 1980s. However, no description of the codimension has so far been given explicitly in terms of the combinatorics of the {barcode}. Our first main theorem presents exactly such a description. 

We define two bars to be an \emph{interacting pair} if they intersect inside $S$ in a prescribed way; see our Definition \ref{defn:interacting.pair}. We then prove

\begin{introthm}[Theorem \ref{thm:main1}]
    Given an equioriented type-A quiver representation, the codimension of its isoclass $\Omega \subset \Rep$ is equal to the number of interacting pairs of bars in its barcode.
\end{introthm}

Since Theorem \ref{thm:CB} guarantees that a barcode exists for \emph{any} pointwise finite-dimensional persistence module $V$, by extension, we naturally apply the conclusion of Theorem \ref{thm:main1} to define the \emph{quiver codimension} $\Qcodim(V)$ to be the number of interacting bars in its barcode; see Definition \ref{defn:Qcodim}.   

Our second main theorem is a precise statement of RT-stability for the codimension. In the statement below, we wait to give the full definitions of \emph{barcode-finite} and \emph{$h$-approximation} until the main body of the paper; see Definitions \ref{defn:barcode.finite} and \ref{defn:h.approx}, respectively.

\begin{introthm}[Theorem~\ref{thm:main2}]
    Consider the barcode-finite $S$-persistence module $V$. Suppose that the finite subset $S' \subseteq S$ admits an $h$-approximation $V'=V|S'$. Provided that $h\geq 2$ 
        \[
        \codim(\Omega') = \Qcodim(V)
        \]
    where $\Omega'$ is the isoclass of the quiver representation associated to $V'$.
\end{introthm}
Theorem \ref{thm:main2} guarantees that the quiver codimension of any persistence module can be effectively computed directly from \emph{quiver representations} by taking appropriate finite sets $S' \subseteq S$. The RT-stability property is: the size of $S'$ may grow to better approximate $S$, and hence the number of nodes of the underlying quiver must grow which, in turn, implies that the ambient representation space $\Rep$ from Equation \eqref{eqn:rep} must grow (in its number of direct summands). Yet, remarkably the \emph{codimension stabilizes} for the corresponding isoclasses $\Omega' \subset \Rep$.

\subsection{Organization of the paper}
In Section \ref{s:PH} we describe important examples of persistence modules---those coming from the persistent homology of data. In Section \ref{s:Q.rep.codim.isoclass} we provide the relevant background on quiver representations and codimension of isoclasses. In Section \ref{s:new.formula.codim} we state and prove our new formula for codimension which is native to the combinatorics of barcodes. In Section \ref{s:PMandQcodim} we define the quiver codimension for general persistence modules. In Sections \ref{s:RTstabQcodim} and \ref{s:RTstabQcodim.proof} we state and prove the RT stability property satisfied by the quiver codimension. In Section \ref{s:ExAndRealizationThms} we prove several realization theorems for the quiver codimension statistic for persistent homology of data. Finally, in Section \ref{s:Discuss.Future} we discuss the qualitative interpretation of the quiver codimension statistic for persistent homology, and we suggest future directions for analysis of quiver codimension specifically and RT-stability generally.

\section{Important classes of examples: persistent homology}
    \label{s:PH}

A \emph{filtration} $Z$ of topological spaces indexed by $S$ is a collection of topological spaces $(Z_s~|~s\in S)$ along with inclusions $Z_s \subseteq Z_t$ whenever $s \leq t$. For any non-negative integer $d$, we can set $V_s = H_d(Z_s,\KK)$ to be the \emph{degree $d$ simplicial homology of the space $Z_s$ with coefficients in $\KK$}. Further, whenever $s\leq t$, we obtain induced mappings in degree $d$ homology $f_{st}: V_s \to V_t$ from the inclusions $Z_s\subseteq Z_t$; see e.g., \cite[Section 2.1]{Hat} as a general reference on simplicial homology. The $S$-persistence module so obtained from this $(V_s)$ and this $(f_{st})$ is called the \emph{(degree $d$) persistent homology} (PH) of the filtration and we will denote it by $V = \PH_d(Z)$; see e.g., \cite{Ghr} as a general reference on persistent homology.

\begin{ex}
	\label{ex:sublevel.filt}
Let $Y$ denote any topological space and let $g:Y \to S$. Form the \emph{closed sublevel sets} $Y_s = \{ y \in Y ~|~ g(y) \leq s\}$. Taken together, the natural nestings $Y_s \subseteq Y_t$ whenever $s\leq t$ are called the \emph{sublevel set filtration of the pair $(Y,g)$}. Denoting this filtration by $Y^g$, we obtain an $S$-persistence module $\PH_d(Y^g)$.
\end{ex}

For the next two examples let $\RR_{\geq 0}$ denote the set of non-negative real numbers and assume $S\subseteq \RR_{\geq0}$. Consider a metric space $(M,\rho)$ and for $s\in\RR_{\geq0}$ and $x\in M$, write $B_s(x) = \left\{ y \in M ~|~ \rho(y,x) \leq s \right\}$ to be the ball centered on $x$ of radius $s$. Let $X = \{x_i\}_{i=1}^N$ denote a finite set of data in $M$; i.e., $X$ is a \emph{point cloud in $M$}.

\begin{ex}
	\label{ex:cech.filt}
For each $s\in S$, the \emph{\v{C}ech complex} is the simplicial complex having a $k$-simplex with vertices $x_{i_0},\ldots,x_{i_k}$ whenever $\bigcap_{j=0}^k B_s(x_{i_j}) \neq \emptyset$. The nerve theorem \cite{Bor} implies that the \v{C}ech complex is homotopy equivalent to the union $C_s = \bigcup_{x \in X} B_s(x)$; see e.g., \cite[Section 1.3]{Ghr}. The filtration given by the inclusions $C_s \subseteq C_t$ is called the \emph{\v{C}ech filtration}. Let us denote this filtration by $\C(X,M)$, and again, we obtain an $S$-persistence module $\PH_d(\C(X,M))$.
\end{ex}

\begin{ex}
	\label{ex:VR.filt}
Alternatively, for each $s \in S$, we can define the \emph{Vietoris-Rips (VR) complex} $X_s$ to be the simplicial complex having a $k$-simplex with vertices $x_{i_0},\ldots,x_{i_k}$ whenever $B_s(x_{i_j}) \cap B_s(x_{i_\ell}) \neq \emptyset$ for each pair $1 \leq j < \ell \leq k$. The nested sequence $X_s \subseteq X_t$ whenever $s\leq t$ defines the \emph{Vietoris-Rips (VR) filtration} which we denote by $\VR(X,M)$. Analogous to the previous examples, $\PH_d(\VR(X,M))$ is an $S$-persistence module. 
\end{ex}

The VR filtration is a computationally efficient alternative to the \v{C}ech filtration since only the pairwise distances between the points need be computed. Fortunately, the VR filtration also provides a controlled approximation to the \v{C}ech filtration; see \cite[Theorem 2.5]{dSG} for a precise formulation.

\section{Quivers, representations, isoclasses, and codimension}
    \label{s:Q.rep.codim.isoclass}

The authors suggest both of the books \cite{DW} and \cite{Sch} as general references for the theory of quivers and their representations. Much of the background material of this section is an amalgam of those sources. 

    \subsection{Quiver generalities}
        \label{ss:quiv.general}

    A \emph{quiver} $Q$ is a directed graph with a finite set of vertices $Q_0$ and a finite set of directed edges $Q_1$. The directed edges are called \emph{arrows}; hence the moniker ``quiver''. For any arrow $a\in Q_1$, we denote the vertex at its \emph{head} by $ha\in Q_0$ and at its \emph{tail} by $ta\in Q_0$. A family of distinguished quivers are the so-called \emph{equioriented type-A quivers}, depicted below
    \[
        1 \longrightarrow 2 \longrightarrow \cdots \longrightarrow n
    \]
where $Q_0 = \{1,2,\ldots,n\}$ and each arrow in $Q_1$ has the form $(i) \to (i+1)$. The name is justified since all arrows point in the same direction; i.e., are \emph{equioriented}, and readers familiar with representation theory of Lie algebras will recognize that these quivers are orientations of the \emph{type-A Dynkin diagram} \cite[Section 11]{Hum}.

\subsection{Representations}

A \emph{representation} of a quiver $Q$, or a \emph{quiver representation}, is a choice of $\KK$-vector spaces $V_i$ at each vertex $i\in Q_0$ along with choices of linear mappings $f_a: V_{ta} \to V_{ha}$ for each arrow $a\in Q_1$. The list $\alpha = (\dim(V_i))_{i\in Q_0}$ is called the \emph{dimension vector} of the quiver representation. We observe that a representation of an equioriented type-A quiver is exactly a pointwise finite-dimensional persistence module over the set $S = \{1 < 2 < \ldots < n\}$.
        
\subsection{Isoclasses and orbits}

Fixing the dimension vector $\alpha$, but allowing the mappings $f_a$ to vary, we form the \emph{ambient space of $\alpha$-dimensional quiver representations}
    \begin{equation}
    	\label{eqn:Rep.d.Q}
		\Rep_\alpha(Q) = \Dirsum_{a\in Q_1} \Hom\left(V_{ta},V_{ha}\right)
    \end{equation}
The algebraic group $G_\alpha = \prod_{i\in Q_0} GL(V_i)$ acts naturally on $\Rep_\alpha(Q)$ by simultaneous change of basis in the head and tail of each arrow; in particular, for $g = (g_i)_{i\in Q_0} \in G_\alpha$ and $f = (f_a)_{a\in Q_1} \in \Rep_\alpha(Q)$ we have the left action
\[
g \cdot f = (g_{ha}f_a g_{ta}^{-1})_{a\in Q_1}.
\]
We will drop the decorations for $\alpha$ and $Q$ from $\Rep_\alpha(Q)$ in the sequel; they will be understood from context. Instead we write, cf.~Equation \eqref{eqn:rep},
	\[
	\Rep = \Dirsum_{i=1}^{n-1} \Hom\left(V_i, V_{i+1}\right).
	\]
It is a classical result of quiver representation theory that two representations are isomorphic (in the category of quiver representations) if and only if they live in the same $G_\alpha$-orbit in $\Rep$ \cite[Corollary 1.3.1]{DW}. Accordingly, we take a $G_\alpha$-orbit in $\Rep$ as the {definition} of an \emph{isomorphism class}, aka \emph{isoclass}, for quiver representations. 

\subsection{Codimension of isoclasses}

Specifically for equioriented type-A representations, the list of persistent Betti numbers $r_{i,j} = \rank(V_i \to V_j)$ for $i<j$ must be constant on any isoclass since basis changes do not alter ranks. Indeed, the isoclasses are in one-to-one correspondence with a choice of persistent Betti numbers \cite{ADF1,ADF2,ADFK}. Taken together, the seminal papers of Abeasis--Del Fra \cite{ADF1,ADF2}, Abeasis--Del Fra--Kraft \cite{ADFK}, and Lakshmibai--Magyar \cite{LM} guarantee that, over any base field $\KK$, the isoclasses in $\Rep$ form algebraic varieties which are Cohen-Macauley, and hence, have pure codimension as subvarieties. We now recall a formula for this codimension, known since the 1980s, in terms of the persistent Betti numbers; see \cite{ADF1}, \cite[Proposition 1]{LM}, and \cite[Equation (1.3)]{B-F}.

\begin{thm} 
	\label{thm:old.codim}
Given an isoclass $\Omega \subset \Rep$, we have
	\[
        \pushQED{\qed}
	\codim(\Omega) = \sum_{i<j} (r_{i,j-1}-r_{i,j})(r_{i+1,j}-r_{i,j}).\qedhere
        \popQED
	\]
\end{thm}
Organize the persistent Betti numbers in a triangular \emph{rank array} \cite[Section 2]{B-F}.
\[
\mqty{
r_{11} && r_{22} && r_{33} && \cdots && r_{nn} \\
& r_{12} && r_{23} && \cdots && r_{n-1,n} & \\
&& r_{13} && \cdots && r_{n-2,n} && \\
&&& \ddots && \stodd &&& \\
&&&& r_{1n} &&&& \\
}
\]
We note that $\dim(V_i) = r_{ii}$ and observe that the rank conditions $r_{i,j} = \rank(V_i \to V_j)$ only make sense and define a meaningful quiver representation if \[r_{i,j}\leq \min(r_{i,j-1},r_{i+1,j})\] for all $i<j$. That is, the triangular rank array must be weakly decreasing when moving downward along all diagonals. If we define the triangle of integers $T_{i,j}$ to be the local portion of the rank array depicting
\[
\mqty{
r_{i,j-1} && r_{i+1,j} \\
& r_{i,j} & 
}
\]
and define the \emph{weight} $w(T) =(m-r)(n-r)$ for any triangle of integers
\[
T = \left(\mqty{
m && n \\
& r & 
}\right)
\]
we see that the rank array allows us to compute the codimension by the formula
\begin{equation}
		\label{eqn:codim.old.again}
\codim(\Omega) = \sum_{i<j} w(T_{i,j}).
\end{equation}

\begin{ex}
    \label{ex:A4.quiver}
Let $Q=A_4$ and consider a representation isomorphic to one having the barcode below.
\begin{center}
    \begin{tikzpicture}
        \node[inner sep = 0pt] at (0,0) (11) {$\bullet$};
        \node[inner sep = 0pt] at (1,0) (12) {$\bullet$};
        \node[inner sep = 0pt] at (2,0) (13) {$\bullet$};

        \node[inner sep = 0pt] at (0,0.5) (21) {$\bullet$};
        \node[inner sep = 0pt] at (1,0.5) (22) {$\bullet$};
        \node[inner sep = 0pt] at (2,0.5) (23) {$\bullet$};

        \node[inner sep = 0pt] at (0,1) (31) {$\bullet$};
        \node[inner sep = 0pt] at (1,1) (32) {$\bullet$};
        \node[inner sep = 0pt] at (2,1) (33) {$\bullet$};
        \node[inner sep = 0pt] at (3,1) (34) {$\bullet$};

        \node[inner sep = 0pt] at (2,1.5) (41) {$\bullet$};
        \node[inner sep = 0pt] at (3,1.5) (42) {$\bullet$};

        \draw (11) -- (12) -- (13);
        \draw (21) -- (22) -- (23);
        \draw (31) -- (32) -- (33) -- (34);
        \draw (41) -- (42);
    \end{tikzpicture}
\end{center}
We can see that the dimension vector of this representation is $\alpha=(3,3,4,2)$, where we think of the dots in each column as representing an ordered basis (from top to bottom) for vector spaces at each vertex. In turn, we have a representation
\[
\KK^3 \to \KK^3 \to \KK^4 \to \KK^2
\]
and the given barcode corresponds to the specific point in the isoclass with maps at each arrow respectively given by the matrices
\[
\mqty[\imat{3}], \qquad 
\mqty[0 & 0 & 0 \\ 1 & 0 & 0 \\ 0 & 1 & 0 \\ 0 & 0 & 1], \qquad
\text{and} \qquad \mqty[1 & 0 & 0 & 0\\ 0 & 1 & 0 & 0].
\]
The rank array (aka persistent Betti numbers) corresponding to the isoclass $\Omega$ to which this representation belongs is
\[
\mqty{
3 && 3 && 4 && 2 \\
& 3 && 3 && 2 & \\
&& 3 && 1 && \\
&&& 1 &&& \\
}
\]
Only the triangle $T_{2,4}$ has a nonzero weight, and we can compute
\[
\codim(\Omega) = w(T_{2,4}) = w\left(\mqty{
3 && 2 \\
& 1 & 
}\right) = (3-1)(2-1) = 2.
\]
\end{ex}

\section{A new formula for the codimension of an isoclass}
    \label{s:new.formula.codim}

For any index set $S$, given an interval $I \subseteq S$ with birth $b$ and death $d$, we will write $I = \langle b,d \rangle $. We adopt the convention, here and throughout the paper, that angle brackets in our interval notation can be interpreted either as inclusion or exclusion of the endpoints. The results that follow are independent of this choice. However, in Section \ref{s:RTstabQcodim.proof}, we will require closed brackets for pointbars $[a,a]$. Further, recall that we will use the words ``bar" and ``interval" interchangeably.  Suppose that $I=\langle b,d \rangle $ and $J=\langle \beta,\delta \rangle$. Without loss of generality, assume that $b \leq \beta$.

\begin{defn}
	\label{defn:interacting.pair}
We say that $I$ and $J$ form an \emph{interacting pair} of bars if either of the following conditions apply.
	\begin{enumerate}[label=(\roman*)]
	\item $b < \beta \leq d < \delta$.
	\item $d < \beta$, but there is no element of $S$ between $d$ and $\beta$.
	\end{enumerate}
In Case (i) we say that $I$ and $J$ are \emph{interlaced}; in Case (ii) we say that $I$ and $J$ are \emph{non-separated}.
\end{defn}

\begin{thm}
	\label{thm:main1}
Given an equioriented type-A quiver representation, the codimension of its isoclass $\Omega \subset \Rep$ is equal to the number of interacting pairs of bars in its barcode.
\end{thm}

\begin{proof}
We observe that for each $s<t$, the persistent Betti numbers have the property that
	\begin{equation}
		\label{eqn:rst}
	r_{st} = \#\{\text{bars~}I \text{~with~} b(I)\leq s \text{~and~} d(I) \geq t\}.
	\end{equation}
and recall that in the case of quiver representations, we consider a pointwise finite-dimensional persistence module over the finite set $S = \{1<\cdots<n\}$ where $n$ is the number of nodes of the quiver. It follows that 
	\begin{enumerate}[label=(\alph*),leftmargin=*]
	\item $(r_{i,j-1}-r_{i,j})$ counts bars with birth $\leq i$ but death \emph{exactly equal to} $j-1$; 
	\item $(r_{i+1,j}-r_{i,j})$ counts bars with birth \emph{exactly equal to} $i+1$ but death $\geq j$.
	\end{enumerate} 
    
First, consider an \emph{interlaced} pair $\langle b,d \rangle$ and $\langle \beta,\delta \rangle$ with $\beta = i+1$ and $d=j-1$. The interlaced criterion guarantees that $b < i+1 \leq j-1 < \delta$. Hence, we must have $j-i \geq 2$ as well as $b \leq i$ and $\delta \geq j$. For a fixed choice of $i$ and $j$, using observations (a) and (b) following Equation \eqref{eqn:rst}, we conclude that the number of pairs of interlaced bars of the form $\langle b, j-1\rangle$ and $\langle i+1,d\rangle$ is the product of the quantity
\[{\left( r_{i,j-1} - r_{i,j} \right)} = {\#\{\text{bars with $b\leq i$,  $d=j-1$}\}}\]
with the quantity
\[{\left( r_{i+1,j} - r_{i,j} \right)} = {\#\{\text{bars with $\beta=i+1$,  $\delta \geq j$}\}}.\]
Observe that this product is exactly $w(T_{i,j})$. In other words, the sum $\sum w(T_{i,j})$ taken over pairs with $j-i \geq 2$ computes the total number of interlaced pairs of bars.

Second, consider a \emph{non-separated} pair $\langle b,d \rangle$ and $\langle \beta,\delta \rangle$. The non-separated criterion implies that $d=k$ and $\beta=k+1$ for some integer $k \in \{1,\ldots,n-1\}$. For this $k$, and again using observations (a) and (b), the number of non-separated pairs of the form $\langle b,k\rangle$ and $\langle k+1,d\rangle$ is the product of the quantity
\[{\left( r_{k,k} - r_{k,k+1} \right)} = \#\{\text{bars with $b\leq k$, $d=k$}\}\]
with the quantity
\[{\left( r_{k+1,k+1} - r_{k,k+1} \right)} = \#\{\text{bars with $\beta=k+1$, $\delta\geq k+1$}\}.\]
This product is $w(T_{k,k+1})$. In other words, the sum $\sum w(T_{i,j})$ taken over pairs with $j-i=1$ computes the total number of non-separated pairs of bars.

Combining these two computations, we obtain that
\begin{align*}
\#\{\text{interacting pairs}\} 
	&= \#\{\text{interlaced pairs}\} + \#\{\text{non-separated pairs}\} \\
	& = \left[\sum_{j-i\geq 2} w(T_{i,j})\right] + \left[\sum_{j-i = 1} w(T_{i,j})\right] = \sum_{i<j} w(T_{i,j}).
\end{align*}
According to Equation \eqref{eqn:codim.old.again}, this last expression is equal to $\codim(\Omega)$.
\end{proof}

\begin{ex}
	\label{ex:codim.A4.with.new.thm}
Recall the barcode from Example \ref{ex:A4.quiver}.
\begin{center}
    \begin{tikzpicture}
        \node[inner sep = 0pt] at (0,0) (11) {$\bullet$};
        \node[inner sep = 0pt] at (1,0) (12) {$\bullet$};
        \node[inner sep = 0pt] at (2,0) (13) {$\bullet$};

        \node[inner sep = 0pt] at (0,0.5) (21) {$\bullet$};
        \node[inner sep = 0pt] at (1,0.5) (22) {$\bullet$};
        \node[inner sep = 0pt] at (2,0.5) (23) {$\bullet$};

        \node[inner sep = 0pt] at (0,1) (31) {$\bullet$};
        \node[inner sep = 0pt] at (1,1) (32) {$\bullet$};
        \node[inner sep = 0pt] at (2,1) (33) {$\bullet$};
        \node[inner sep = 0pt] at (3,1) (34) {$\bullet$};

        \node[inner sep = 0pt] at (2,1.5) (41) {$\bullet$};
        \node[inner sep = 0pt] at (3,1.5) (42) {$\bullet$};

        \draw (11) -- (12) -- (13);
        \draw (21) -- (22) -- (23);
        \draw (31) -- (32) -- (33) -- (34);
        \draw (41) -- (42);
    \end{tikzpicture}
\end{center}
We see only two pairs of interacting bars. Labeling from the top to bottom, the first/third bars and the first/fourth bars are interlaced. Hence, as in Example \ref{ex:A4.quiver}, we conclude that $\codim(\Omega) = 2$.
\end{ex}

     In the quiver representation literature, the barcode is called a \emph{lacing diagram}. In a seminal 2005 paper \cite{KMS}, Knutson--Miller--Shimozono introduced the notion of an \emph{extended} lacing diagram and showed that the codimension of an isoclass could be computed by counting crossing strands in a \emph{minimal} extension; we will not define the minimality condition in the present paper, but see \cite[Theorem~3.8]{KMS} and \cite[Section~2]{BFR} for details. While the Knutson--Miller--Shimozono formulation is combinatorial, our Theorem \ref{thm:main1} requires \emph{only} the knowledge of the birth and death times to compute the codimension, and not any knowledge of a special extension of the barcode. Nevertheless, it would be interesting to compare our Theorem \ref{thm:main1} with the construction of a minimal extended lacing diagram; however, this setting is too general for the TDA framework and investigating it is beyond the scope of the current paper.

\section{Persistence modules and Qcodim}
    \label{s:PMandQcodim}

Let $V$ denote a pointwise finite-dimensional $S$-persistence module. The work of Crawley-Boevey \cite[Theorem 1.1]{CB}, which we have recalled in the present paper as Theorem \ref{thm:CB} of the introduction, implies that
\begin{equation}
	\label{eqn:V.interval.decomp}
V \cong \Dirsum_{I \subseteq S} \KK_I^{m_I}
\end{equation}
where the sum is taken over intervals $I \subseteq S$, $\KK_I$ denotes the corresponding interval module, and $m_I$ denotes the multiplicity of $\KK_I$ in the decomposition. 

\begin{defn}
	\label{defn:barcode.finite}
A pointwise finite-dimensional $S$-persistence module $V$ is \emph{barcode-finite} if $\sum_{I\subseteq S} m_I < \infty$.
\end{defn}

In other words, $V$ is \emph{barcode-finite} when the number of interval modules in its decomposition (counted with multiplicity) is finite; i.e., if the total number of \emph{bars} in its \emph{barcode} is finite. Even if $V$ is a persistence module over an infinite indexing set $S$, as long as $V$ is barcode-finite, we use the language of Theorem \ref{thm:main1} to extend the idea of codimension to this more general setting.

\begin{defn}
	\label{defn:Qcodim}
Let $V$ be a barcode-finite $S$-persistence module. The \emph{quiver codimension} of $V$, denoted $\Qcodim(V)$, is the number of interacting pairs of bars in the barcode of $V$.
\end{defn}

Definitions \ref{defn:barcode.finite} and \ref{defn:Qcodim} ensure that the value of $\Qcodim(V)$ is always a finite non-negative integer. We have adopted these definitions because, in the sequel of the present paper, we show how to effectively determine $\Qcodim$ via codimensions for \emph{bona fide} {quiver representations}, which are always barcode-finite. However, we comment that we have left open the following possible generalizations, in two directions, for future work. 
	\begin{enumerate}[label=(\arabic*),leftmargin=*]
	\item The barcode has infinitely-many bars but only finitely many interacting pairs, and hence $\Qcodim(V)$ is still finite;
	\item the barcode happens to admit infinitely-many interacting pairs and we simply allow $\Qcodim(V) = \infty$.
	\end{enumerate}

\section{Representation theoretic stability of Qcodim}
    \label{s:RTstabQcodim}

We view the statement of Theorem \ref{thm:main2} (see below) as the precise formulation of \emph{RT-stability} for quiver codimension of persistence modules.

\begin{defn}
	\label{defn:crit.pts}
Let $V$ be a barcode-finite $S$-persistence module. Define the set of \emph{critical points} for $V$ to be subset  $Z(V)\subseteq S$ given by the union of all births and all deaths of all bars in the barcode. By definition, $Z(V)$ is a finite set.
\end{defn}

\begin{defn}
	\label{defn:h.approx}
Let $h$ be a non-negative integer. Let $V$ be a barcode-finite $S$-persistence module with critical points $Z(V) = \{ z_1 < \cdots < z_m\}$. A finite subset $S' \subseteq S$ admits an \emph{$h$-approximation $V' = V|S'$ to $V$} if
	\begin{enumerate}[leftmargin=*,label=(\roman*)]
	\item there exists $s,t \in S'$ with $s \leq z_1$ and $z_m \leq t$,
	\item whenever $\langle z_k,z_i \rangle$ and $\langle z_{i+1},z_\ell \rangle$ is a non-interacting pair of bars in the barcode for $V$, there exists $s\in S'$ with $z_i < s < z_{i+1}$, and
 	\item $\# \left( S ' \cap \langle z_i,z_{i+1} \rangle\right) \geq h$ for all $i=1,\ldots,m-1$,
	\end{enumerate}
where the intervals in (ii) and (iii) are taken as subsets of $S$.
\end{defn}

    Item (iii) in Definition \ref{defn:h.approx} is the key condition in the sequel; it ensures that $S'$ is ``dense enough" to detect  important features of the barcode. On the other hand, conditions (i) and (ii) address special cases in our analysis. In particular, (i) ensures that our proof of Theorem \ref{thm:main2} will go through when $z_1$ is a \emph{death} time or $z_m$ is a \emph{birth} time. This can occur (say if $S = \RR$) when $V$ has bars of the form $\langle -\infty,z_1 \rangle$ or $\langle z_m,\infty \rangle$. Condition (ii) ensures that non-interacting pairs of bars in $V$ do not restrict to interacting pairs in $V'$; we will need this in the proofs of Lemmas \ref{lem:non.interact.pb.to.no.interact} and \ref{lem:non.interact.to.non.interact}.  

\begin{thm}
	\label{thm:main2}
Consider the barcode finite $S$-persistence module $V$. Suppose that the finite subset $S' \subseteq S$ admits an $h$-approximation $V'=V|S'$. Provided that $h\geq 2$ \[\codim(\Omega') = \Qcodim(V)\]
where $\Omega'$ is the isoclass of the quiver representation associated to $V'$. 
\end{thm}

In other words, the barcodes of $V$ and $V'$ have the same number of interacting pairs. The proof of this theorem is the subject of the next section.

The content of Theorem \ref{thm:main2} is twofold. First, even if the critical points $Z(V)$ can not be computed exactly, it is enough to take a $2$-approximation of $V$ and compute the codimension from the resulting quiver representation. Second, the key to the proof is that while the \emph{flavors} of interacting pairs in $V$ and $V'$ may differ---i.e., the number of non-separated vs.~interlaced pairs \emph{need not} be constant---the \emph{total number} of interacting pairs \emph{is} constant.

\begin{rem}
When $S = \RR$ (or some connected subset of $\RR$), an $S'$ which admits an $h$-approximation can always be constructed by knowing
\begin{itemize}[leftmargin=*]
    \item bounds on the set of critical points; say $m\leq |z| \leq M$ for all $z\in Z(V)$, and
    \item a lower bound on the distance between critical points; $\min_{i<j}\{|z_i - z_j|\} \geq \epsilon$.
\end{itemize}
Then, the finite set $S'$ formed by the equally spaced sequence \[m, m+\delta, m+2\delta, \ldots , m + N\delta\] until $m+N\delta \geq M$ admits an $h$-approximation provided $\delta\leq\epsilon/h$.
\end{rem}

\section{Proof of Theorem \ref{thm:main2}}
    \label{s:RTstabQcodim.proof}

Throughout the section, we take as blanket hypotheses those of Theorem \ref{thm:main2} on $V$, $V'$, $S$, $S'$, and $h$; i.e., $h\geq 2$ and $V'=V|S'$ is an $h$-approximation of the barcode-finite $S$-persistence module $V$.

\subsection{Interactions and restrictions of pointbars}
    \label{ss:pointbars}
Recall that a pointbar in $S$ (respectively $S'$) is a bar $I$ such that $b(I)=a=d(I)$ for some $a\in S$ (resp.~$a\in S')$. We will write such a pointbar as $[a]$.
\begin{lem}
	\label{lem:pointbar.interact.nonsep}
If a pointbar is part of an interacting pair, the pair must be non-separated. 
\end{lem}

\begin{proof}
Denote the pointbar by $[a]$ and the other bar (possibly a pointbar, or not) by $\langle b,d \rangle$. If the pair is interlaced, we must either have $a < b \leq a$ or $a \leq d < a$ which are both impossible.
\end{proof}

\begin{lem}
	\label{lem:pointbar.nonsep.to.pointbar.nonsep}
Suppose that $I = [a] $ is a pointbar for $V$ which interacts with a bar $J$ in the barcode for $V$. Then we must have that $a \in S'$, $I' = [a] $ is a pointbar for $V'$, and $I'$ interacts with a corresponding bar $J'$ in the barcode for $V'$. Moreover, the bars $I$ and $I'$ have the same multiplicity in their respective barcodes, and the bars $J$ and $J'$ have the same multiplicity in their respective barcodes.
\end{lem}

\begin{proof}
Lemma \ref{lem:pointbar.interact.nonsep} ensures that $I$ and $J$ must interact as a non-separated pair. Write $J = \langle b,d \rangle $ where it is possible that $b=d$ and so $J$ is also a pointbar, and first consider the case that $a<b$. Since $I$ and $J$ are non-separated, it must be the case that there are no elements of $S$ between $a$ and $b$. Since $h\geq 2$, we must have that $\langle a,b \rangle \cap S' = \{a,b\}$. Taken together, the previous two statements imply that $a,b\in S'$ and there are no elements of $S'$ between $a$ and $b$. Hence, $I' = [a] $ is a bar for $V'$ and there must exist a bar $J' = \langle b,d' \rangle $ with $b \leq d' \leq d$ with which it interacts. The case $d<a$ is similar. The statement on multiplicities follows from the fact that $V' = V|S'$.
\end{proof}

\begin{lem}
	\label{lem:non.interact.pb.to.no.interact}
Suppose that $I = [a] $ is a pointbar for $V$ which does not interact with any other bars in the barcode for $V$. One of the following must hold:
	\begin{itemize}
	\item $a\notin S'$ and thus $I$ does not restrict to a pointbar for the restriction $V'$;
	\item $a \in S'$, and thus $I = [a] $ is a pointbar for $V'$ but still does not interact with any other bars for $V'$.
	\end{itemize}
\end{lem}

\begin{proof}
We need only prove the claim in the second mutually exclusive bullet point. Suppose that $J = \langle b,d \rangle $ is a bar in $V$. We consider the case that $a<b$; the others are analogous. Condition (ii) of Definition \ref{defn:h.approx} ensures there is an $s'\in S'$ with $a < s' < b$. If $J$ is a pointbar which does not restrict to $V'$, there is nothing to prove. However, if $b\in S'$ and $J$ is a pointbar, or $J$ is not a pointbar but restricts to $J' = \langle b',d' \rangle $, in either case we have $a < s' < b \leq b'$ and hence $I$ interacts with neither the pointbar $J$ nor $J'$ in the barcode for $V'$.
\end{proof}

\subsection{Interactions and restrictions for bars which are not pointbars}
    \label{ss:nonpointbars}

Let $\BB$ denote the set of bars for $V$ which are not pointbars. Analogously, let $\BB'$ denote the set of bars for $V'$ which are not pointbars. We define a map $\Phi:\BB \to \BB'$ given as follows. For $I=\langle b,d\rangle \in \BB$ write $I'=\Phi(I) = \langle b',d'\rangle \in \BB'$ for some $b',d'\in S'$ and satisfying the crieteria that
    \begin{enumerate}[label=(C\arabic*)]
        \item $b\leq b'$ and no other elements of $S'$ are between $b$ and $b'$;
	\item $d'\leq d$ and no other elements of $S'$ are between $d'$ and $d$.
    \end{enumerate}

\begin{lem}
	\label{lem:I.admits.I'}
The map $\Phi$ is well-defined and is a one-to-one correspondence. Further, the multiplicity for any $I\in \BB$ in the barcode for $V$ equals the multiplicity of $I' = \Phi(I)$ in the barcode for $V'$. 
\end{lem}

\begin{proof}
Set $b' = \inf\{ s \in S' ~|~ s\geq b\}$ and $d' = \sup\{ s \in S' ~|~ s \leq d\}$. Since $S'$ is at least a $2$-approximation of $S$; i.e., $h\geq 2$, we are assured that $b'$ and $d'$ both exist, and moreover, $b'<d'$. Since $V'$ is a restriction of $V$, $\langle b',d' \rangle$ will be in the barcode for $V'$ and satisfy (C1) and (C2) by construction. The statement on equality of multiplicities is guaranteed by the fact that $V' = V|S'$.
\end{proof}

\begin{lem}
	\label{lem:IJ.interlaced.I'J'.interlaced}
Suppose that $I$ and $J$ are an interlaced pair of bars for $V$. By Lemma \ref{lem:pointbar.interact.nonsep}, both $I,J\in\BB$. The corresponding bars $I'=\Phi(I)$ and $J'=\Phi(J)$ interact in the barcode for $V'$. 
\end{lem}

\begin{proof}
Write $I = \langle b,d \rangle $ and $J = \langle \beta,\delta \rangle $. Without loss of generality, the interlaced criterion says that $b < \beta \leq d < \delta$.

\underline{Case I: $\beta < d$}. Let $I' = \langle s_1,s_3 \rangle $ and $J' = \langle s_2,s_4 \rangle $. Since each of $b < \beta < d < \delta$ are critical points in $Z(V)$ and $h \geq 2$, we can conclude
	\begin{itemize}
	\item $b \leq s_1< \beta$,
	\item $\beta\leq s_2 < s_3 \leq d$, and
	\item $d < s_4 \leq \delta$.
	\end{itemize}
Putting these together, we must have that $s_1 < s_2 < s_3 < s_4$. Hence $I' = \langle s_1,s_3 \rangle $ and $J' = \langle s_2,s_4 \rangle $ are an interlaced pair of bars for $V'$.

\underline{Case II: $\beta = d$}. Write $I' = \langle t_1, t_2 \rangle $ and $J' = \langle t_3,t_4 \rangle$ for $t_1<t_2\leq t_3 < t_4$ some elements of $S'$. Since $h \geq 2$ we can conclude that $t_1, t_2 \in \langle b,d \rangle $ and $t_3,t_4 \in \langle \beta,\delta \rangle $. If $\beta = d$ happens to be in $S'$, then $t_2 = t_3$ and $I'$ and $J'$ are again interlaced. On the other hand, if $\beta = d$ is not an element of $S'$, since $t_2$ is maximal among elements of $S'$ bounded above by $d$ and $t_3$ is minimal among elements of $S'$ bounded below by $\beta$, we conclude that there are no elements of $S'$ between $t_2$ and $t_3$. Hence, $I'$ and $J'$ form a non-separated pair of bars for $V'$.

In either Case I or II, we conclude that the bars $I'$ and $J'$ interact.
\end{proof}

\begin{ex}
As an example of the scenario in Case II from the preceding proof, consider the figure below. We depict the barcode for a persistence module $V$ over $S=\RR$ with 2 bars. The elements of the set $Z(V)$ of critical points of $V$ are depicted on the number line at the bottom as blue dots in $S=\RR$. A $3$-approximation $S'$ is depicted as a subset of the number line with red dots. The barcode of the restriction $V'=V|S'$ is also shown.
   \begin{center}
       
        \begin{tikzpicture}

    \node[inner sep=0pt] (bb1) at (-2.3,1.2) {$\bullet$};
    \node[inner sep=0pt] (dd1) at (1.6,1.2) {$\bullet$};
    \node[inner sep=0pt] (bb2) at (1.6,0.8) {$\bullet$};
    \node[inner sep=0pt] (dd2) at (4.2,0.8) {$\bullet$};

    \node at (-3.2,1) {$V$:};
    \node at (-3.2,0) {$V'$:};

   \node[inner sep=0pt] (b1) at (-2,0) {$\bullet$};
    \node[inner sep=0pt] (d1) at (1,0) {$\bullet$};
    \node[inner sep=0pt] (b2) at (2,0) {$\bullet$};
    \node[inner sep=0pt] (d2) at (4,0) {$\bullet$};

    \foreach \x in {-2,...,4} 
               \node at (\x,0) {$\bullet$};

    \foreach \x in {-2,...,4}
                \node at (\x,-1) {{\color{red}$\bullet$}};

    \node at (-2.5,-1) {{\color{red}$\bullet$}};
    \node at (4.4,-1) {{\color{red}$\bullet$}};

    \node (z1) at (-2.3,-1) {{\color{blue}$\bullet$}};
    \node (z2) at (1.6,-1) {{\color{blue}$\bullet$}};
    \node (z3) at (4.2,-1) {{\color{blue}$\bullet$}};

    \draw (b1)--(d1);
    \draw (b2)--(d2);
    \draw (bb1)--(dd1);
    \draw (bb2)--(dd2);

    \draw[->] (-3,-1)--(5,-1);
    \node (S) at (-3.2,-1) {$S$};
    
  \end{tikzpicture}
   \end{center}
The two bars from $V$ interact as an interlaced pair, while the corresponding bars for $V'$ (images of the map $\Phi$) interact as a non-separated pair.
\end{ex}

\begin{lem}
	\label{lem:nonsep.to.nonsep}
Suppose that $I,J \in \BB$ are a non-separated pair of bars for $V$. Then the corresponding bars $I'=\Phi(I)\in\BB'$ and $J'=\Phi(J)\in\BB'$ are a non-separated pair of bars for $V'$.
\end{lem}

\begin{proof}
Write $I = \langle b,d \rangle $ and $J=\langle \beta,\delta \rangle $ and without loss of generality assume that $b<d<\beta<\delta$. The non-separated criterion implies that there are no elements of $S$ between $d$ and $\beta$. Since $S' \subseteq S$, we moreover conclude that there are no elements of $S'$ between $d$ and $\beta$. Then, writing $I' = \langle b',d' \rangle $ and $J' = \langle \beta',\delta' \rangle $, the criteria (C1) and (C2) imply that there are no elements of $S'$ between $d'$ and $d$, nor between $\beta$ and $\beta'$. We deduce there are no elements of $S'$ between $d'$ and $\beta'$; i.e., $I'$ and $J'$ are also non-separated in the barcode for $V'$.
\end{proof}

\begin{lem}
	\label{lem:non.interact.to.non.interact}
Suppose that $I,J\in\BB$ are not an interacting pair of bars for $V$. Then the corresponding bars $I'=\Phi(I)$ and $J'=\Phi(J)$ are not an interacting pair for $V'$.
\end{lem}

\begin{proof}
Write $I = \langle b,d \rangle $, $J = \langle \beta,\delta \rangle $, $I' = \langle b',d' \rangle $, and $J' = \langle \beta',\delta' \rangle $. If $I$ and $J$ do not interact then, without loss of generality, one of the following cases must hold:
	\begin{enumerate}[label = (\roman*)]
	\item $b = \beta$;
	\item $d = \delta$;
	\item $b < \beta < \delta < d$;
	\item $b < d < s < \beta < \delta$ for some $s\in S$.
	\end{enumerate}
In case (i), the criteria (C1) and (C2) guarantee that $b' = \beta'$. Similarly, in case (ii), $d' = \delta'$. In either case, $I'$ and $J'$ do not form an interacting pair. 

In case (iii), the criteria (C1) and (C2) ensure that $b'\leq \beta'$ and $\delta' \leq d'$. The condition that $h\geq 2$ ensures that $\beta' < \delta'$ and we have $b'\leq \beta'<\delta'\leq d'$ which is sufficient 
to conclude that $I'$ and $J'$ do not form an interacting pair.

In case (iv), the criteria (C1) and (C2) imply that $b' < d' $ and $\beta' < \delta'$. Condition (ii) of Definition \ref{defn:h.approx} ensures that we must have $s' \in S'$ so that $b'<d'<s'<\beta'<\delta'$. Hence, we have an analogous string of inequalities in $S'$ to that in $S$, and we conclude that $I'$ and $J'$ do not form an interacting pair.
\end{proof}

\subsection{Finishing the proof of Theorem \ref{thm:main2}}
On the one hand, Lemmas \ref{lem:pointbar.nonsep.to.pointbar.nonsep}, \ref{lem:IJ.interlaced.I'J'.interlaced}, and \ref{lem:nonsep.to.nonsep} together ensure that every interacting pair of bars for $V$ restricts to an interacting pair of bars for $V'$. On the other hand, Lemmas \ref{lem:non.interact.pb.to.no.interact} and \ref{lem:non.interact.to.non.interact} ensure that every non-interacting pair of bars for $V$ either restricts to a non-interacting pair of bars in $V'$, or at least one of the bars does not restrict at all to the barcode for $V'$, and hence does not contribute any new interactions. This proves that the number of interacting pairs of bars for $V$ is equal to the number of interacting pairs of bars for $V'$, and we conclude that $\Qcodim(V) = \Qcodim(V') = \codim(\Omega')$ as desired. \qed

\section{Examples and Realization Theorems}
    \label{s:ExAndRealizationThms}

\subsection{Qcodim from VR or \v{C}ech filtrations}
Throughout this subsection, we consider the metric space $\RR^p$ with standard Euclidean metric. When $X = \{x_i\}_{i=1}^N \subset \RR^p$ is a finite point cloud, recall from Examples \ref{ex:cech.filt} and \ref{ex:VR.filt} we may define persistence modules over $S = \RR_{\geq0}$ by taking the persistent homologies
	\[
	\PH_d(\VR(X,\RR^p)) \quad\text{or}\quad \PH_d(\C(X,\RR^p)).
	\]

\begin{thm}	
	\label{thm:VR.cech.0}
For any point cloud $X \subset \RR^p$, in homological degree $d=0$ we have
	\[
	\Qcodim(\PH_0(\VR(X,\RR^p))) = 0 \quad\text{and}\quad \Qcodim(\PH_0(\C(X,\RR^p))) = 0.
	\]
\end{thm}

\begin{proof}
Recall $0$-cycles in simplicial homology correspond to connected components. For $s=0$, in both the VR and \v{C}ech complex, every individual point of $X$ corresponds to the birth of a $0$-cycle. As $s$ increases, connected components merge and $0$-cycles die, but no new $0$-cycles are born. Thus, every pair of bars in the resulting barcode will have the \emph{same} birth time; namely $s=0$. Hence, there are no interacting pairs of bars and $\Qcodim = 0$ in both cases.
\end{proof}

We will now construct a family of point clouds $X_c \subset \RR^2$ indexed by a non-negative integer $c$ and an increasing sequence of real numbers $\epsilon_0 < \epsilon_1 < \cdots < \epsilon_{c+1}$. In the following, we will consider the barcode for degree $d=1$ persistent homology applied to the VR filtration of $X_c$ over $S=\RR_{\geq0}$, and in particular, the numbers $\epsilon_i$ (for $1\leq i\leq c+1$) will be birth times of the desired bars and $\sqrt{2}\,\epsilon_i$ (for $0\leq i \leq c$) will be the death times. See Figures \ref{fig:Rips} and \ref{fig:barcodes} for examples. We further recall $1$-cycles in simplicial homology correspond to graph cycles in the $1$-skeleton of a simplicial complex which are not boundaries of any $2$-simplices.

\begin{figure}
    \centering
        \framebox{\includegraphics[width=0.95\textwidth]{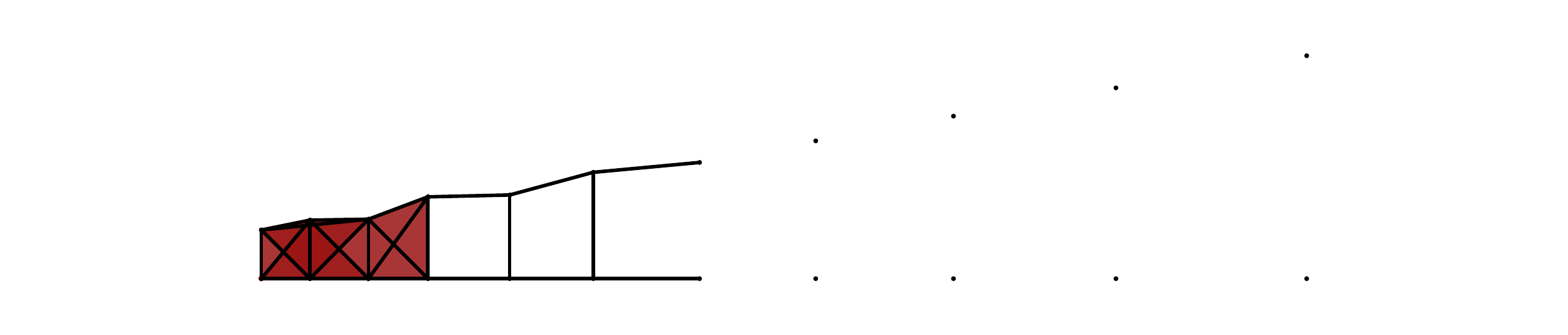}}
        \framebox{\includegraphics[width=0.95\textwidth]{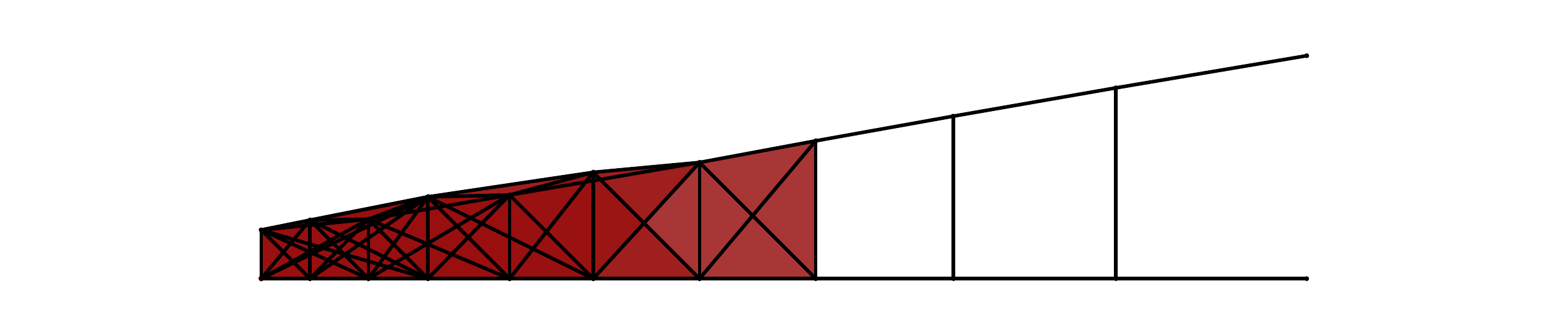}}
    \caption{The underlying point cloud has the form $X_c$ with $c=8$. Top: 2-skeleton of the VR complex for $s=4.5$. Bottom: 2-skeleton of the VR complex for $s=8$. Images created with the software \cite{FiltrationDemos}.}
    \label{fig:Rips}
\end{figure}

\begin{figure}
    \centering
    \begin{minipage}{0.49\textwidth}
        \includegraphics[width=0.95\textwidth]{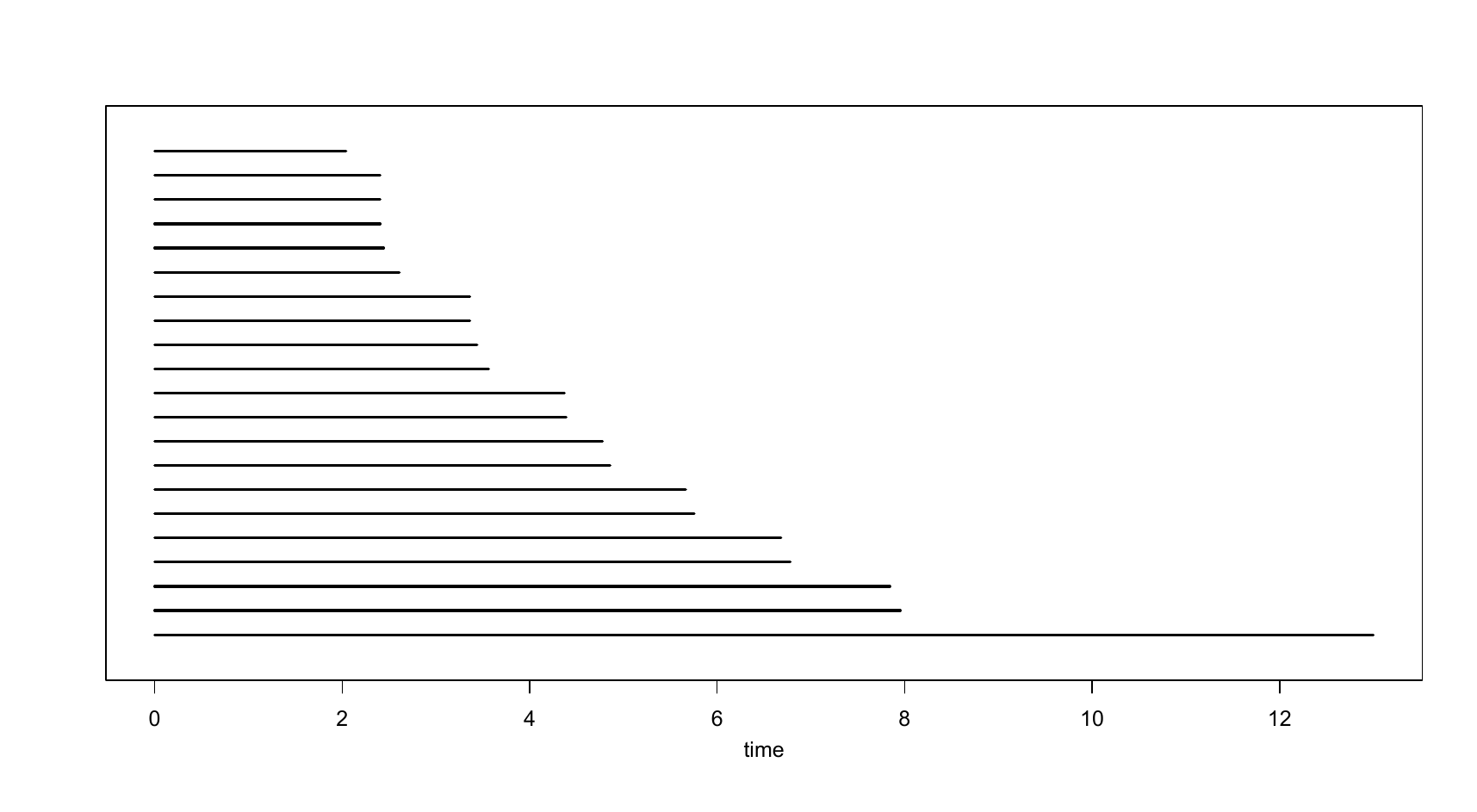}
    \end{minipage}
    \begin{minipage}{0.49\textwidth}
        \includegraphics[width=0.95\textwidth]{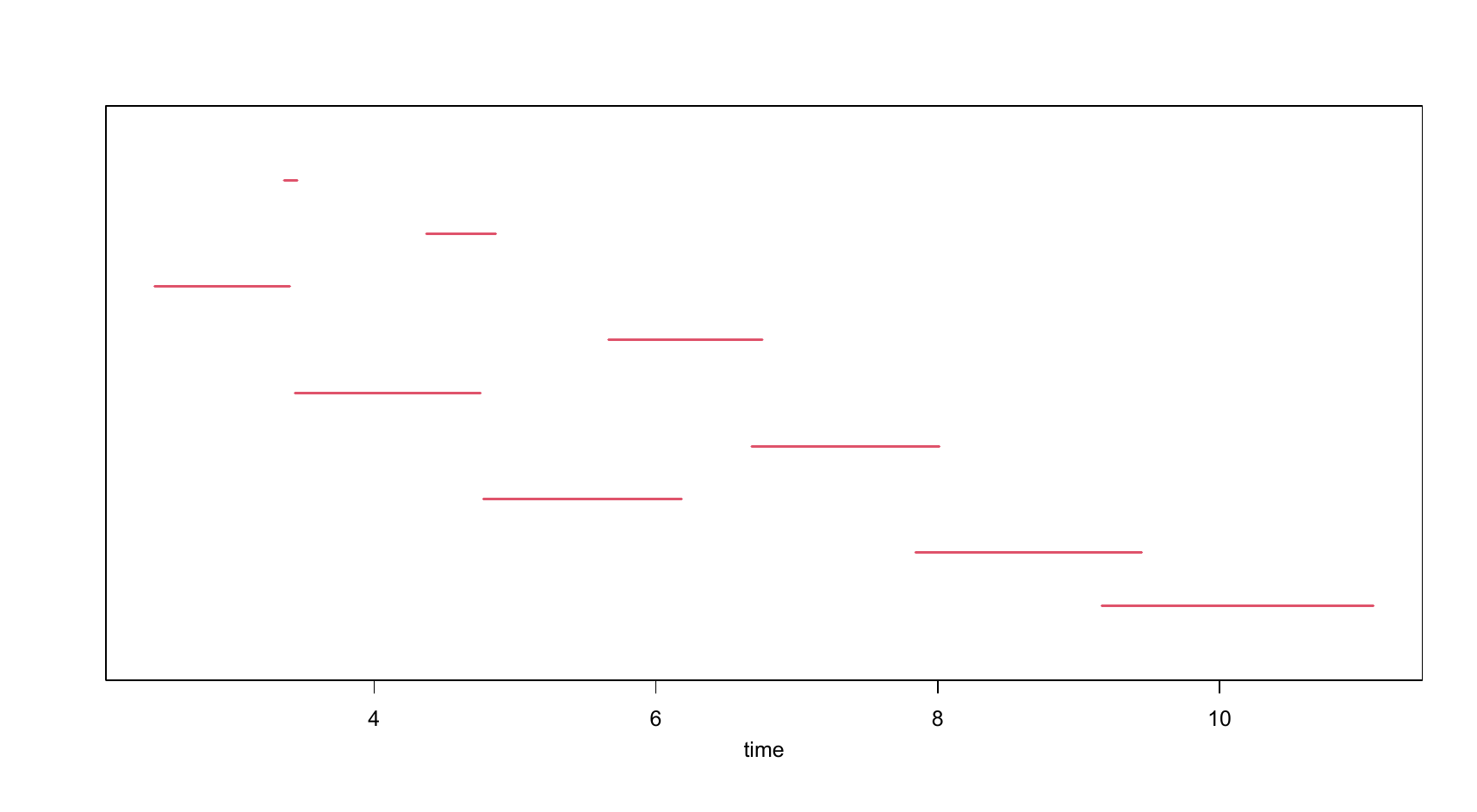}
    \end{minipage}
    \caption{The $d=0$ (left) and $d=1$ (right) barcodes for $\PH_d(\VR(X_8,\RR^2))$ with $X_8$ as in Figure \ref{fig:Rips}. The axis labeled ``time" denotes values from the indexing set $S = \RR_{\geq0}$. Notice that the $d=0$ barcode is an example of the result of Theorem \ref{thm:VR.cech.0}. The images were produced using the TDA package library and GUDHI algorithm in the statistical software, R.}
    \label{fig:barcodes}
\end{figure}

Set $\epsilon_0=1$. We will start with four points in $X_c$: $(0,0)$, $(0,2)$, $(2,0)$, $(2, 2\epsilon_1)$, where $\epsilon_1 \in (\epsilon_0, \sqrt{2}\epsilon_0) = (1,\sqrt{2})$. A $1$-cycle $I_0$ will be born at $\epsilon_1$ and die at $\sqrt{2}\epsilon_0$ in $S$; see the example in Figure \ref{fig:first.cycle}. If $c=0$, we stop here and we have formed $X_0$ with barcode having the single bar $I_0$.
    
If $c\geq 1$, we continue by adding two more points to $X_c$: $(2+2\epsilon_1, 2\epsilon_2)$, $(2+2\epsilon_1, 0)$ for some choice of $\epsilon_2 \in (\epsilon_1,\sqrt{2}\epsilon_0)$. Observe that a new $1$-cycle $I_1$ will be born at $\epsilon_2$ and die at $\sqrt{2}\epsilon_1$. By construction, this guarantees that $I_0$ is interlaced with $I_1$.
    
If $c \geq 2$, 
%
we inductively continue for each integer $n$ in the range $3 \leq n \leq c+1$, by adding two more points to $X_c$: $(2\sum_{i=0}^{n-1}\epsilon_i,0)$ and $(2\sum_{i=0}^{n-1}\epsilon_i,2\epsilon_n)$ with $\epsilon_n \in (\sqrt{2}\epsilon_{n-3},\sqrt{2}\epsilon_{n-2}) $. For each such $n$, we contribute a new $1$-cycle $I_{n-1}$ having birth time $\epsilon_n$ and death time $\sqrt{2}\epsilon_{n-1}$. Again, by construction $I_{n-1}$ is interlaced with $I_{n-2}$, but does not interact with any $I_j$ for $j < n-2$.


\begin{figure}
    \centering
    \begin{minipage}{0.45\textwidth}
\begin{center}
 \begin{tikzpicture}
    \node[inner sep=0pt] (A)  [label=below: {$(0,0)$}] at (0,0) {$\bullet$};
    \node[inner sep=0pt] (B) [label=below: {$(2,0)$}] at (2,0) {$\bullet$};
    \node[inner sep=0pt] (C) [label=above: {$(0,2)$}] at (0,2) {$\bullet$};
    \node[inner sep=0pt] (D) [label=above: {$(2,2.6)$}]at (2,2.6) {$\bullet$};

    \coordinate (a) at (0,0);
    \coordinate (b) at (2,0);
    \coordinate (c) at (0,2);
    \coordinate (d) at (2,2.6);

    \draw[red] (A) circle (1.3);
    \draw[red] (B) circle (1.3);
    \draw[red] (C) circle (1.3);
    \draw[red] (D) circle (1.3);

    \draw (A) -- (B) -- (D) -- (C) -- (A);
    \draw[dotted] (B) -- (C);
    \draw[dotted] (A) -- (D);

\end{tikzpicture}
   
\end{center}
\end{minipage}
\begin{minipage}{0.45\textwidth}
\begin{center}
 \begin{tikzpicture}
    \node[inner sep=0pt] (A)  [label=below: {$(0,0)$}] at (0,0) {$\bullet$};
    \node[inner sep=0pt] (B) [label=below: {$(2,0)$}] at (2,0) {$\bullet$};
    \node[inner sep=0pt] (C) [label=above: {$(0,2)$}] at (0,2) {$\bullet$};
    \node[inner sep=0pt] (D) [label=above: {$(2,2.6)$}]at (2,2.6) {$\bullet$};

    \coordinate (a) at (0,0);
    \coordinate (b) at (2,0);
    \coordinate (c) at (0,2);
    \coordinate (d) at (2,2.6);

    \filldraw[fill=black!5] (a) -- (b) -- (d) -- (c) -- cycle;
    \draw (B) -- (C);
    \draw[dotted] (A) -- (D);

    \draw[red] (A) circle (1.415);
    \draw[red] (B) circle (1.415);
    \draw[red] (C) circle (1.415);
    \draw[red] (D) circle (1.415);

\end{tikzpicture}
   
\end{center}
\end{minipage}
    \caption{The depicted data points form the initial $1$-cycle in $X_c$ (here with $\epsilon_1=1.3$). The red circles are boundaries of the balls used to construct the VR complex $(X_c)_{s}$ with $s=\epsilon_1 = 1.3$ on the left and $s=\sqrt{2}$ on the right. The solid black edges depict the $1$-skeleton of the VR complex; dotted lines depict edges not yet included in the VR complex. Filled gray triangles depict the 2-skeleton of the VR complex. Hence, the left picture depicts the birth of the initial $1$-cycle while the right depicts its death.}
    \label{fig:first.cycle}
\end{figure}

\begin{thm}
	\label{thm:VR.1}
$\Qcodim(\PH_1(\VR(X_c,\RR^2))) = c$.
\end{thm}

\begin{proof}
By construction of $X_c$, the resulting VR filtration will admit a persistent homology barcode in degree $d=1$ with bars $I_0, I_1, \ldots, I_c$ with the property that the interacting pairs are exactly the pairs of the form $(I_{k-1}, I_k)$ for $k = 1,\ldots,c$. The result follows from Definition \ref{defn:Qcodim}.
\end{proof}

Taken together the results of Theorems \ref{thm:VR.cech.0} and \ref{thm:VR.1} say that, while degree $d=0$ persistent homology is uninteresting for VR filtrations from the perspective of the quiver codimension statistic, already in degree $d=1$ every possible non-negative integer can be realized as a quiver codimension.

\subsection{Qcodim from sublevel set filtrations}
In this subsection, we will recall the sublevel set filtrations of Example \ref{ex:sublevel.filt}. In particular, we will construct a family of subsets of $\RR^2$ for which the degree $d=0$ persistent homology of the sublevel set filtration has quiver codimension which can take \emph{any} non-negative integer value. In other words, the quiver codimension of degree $d=0$ persistent homology may still admit interesting information on a filtration outside of the special cases of VR and \v{C}ech filtrations; cf.~Theorem \ref{thm:VR.cech.0}.

Let $c$ be a non-negative integer. Form $c+2$ points $A_i = (i,i) \in \RR^2$ for each integer $i$ in the range $-1\leq i \leq c$. Similarly, form $c+2$ points $B_i = (i,i+2) \in \RR^2$ for $-1\leq i \leq c$. Let $\Sigma_c$ denote the piecewise linear \emph{sawtooth space} obtained by connecting the $2c+4$ points $\{A_i\}\cup \{B_i\}$ via straight line segments in the sequence 
\[A_{-1} \to B_{-1} \to A_0 \to B_0 \to \cdots \to A_c \to B_c \]
see Figure \ref{fig:sawtooth2}. 
\begin{figure}
	\centering
	\begin{minipage}{0.45\textwidth}
	\includegraphics[width = 0.8\textwidth]{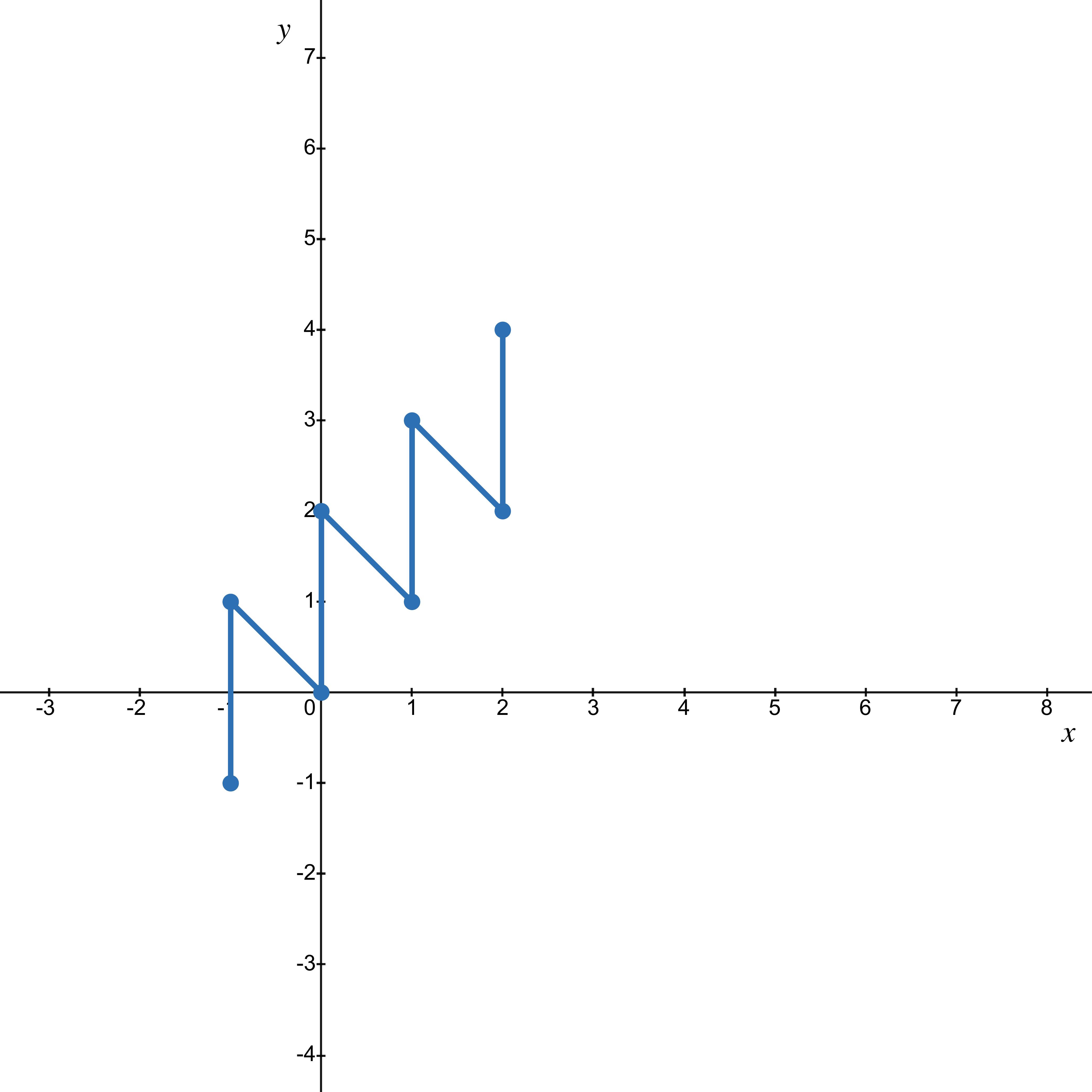}
	\end{minipage}
	\begin{minipage}{0.45\textwidth}
	\includegraphics[width = 0.8\textwidth]{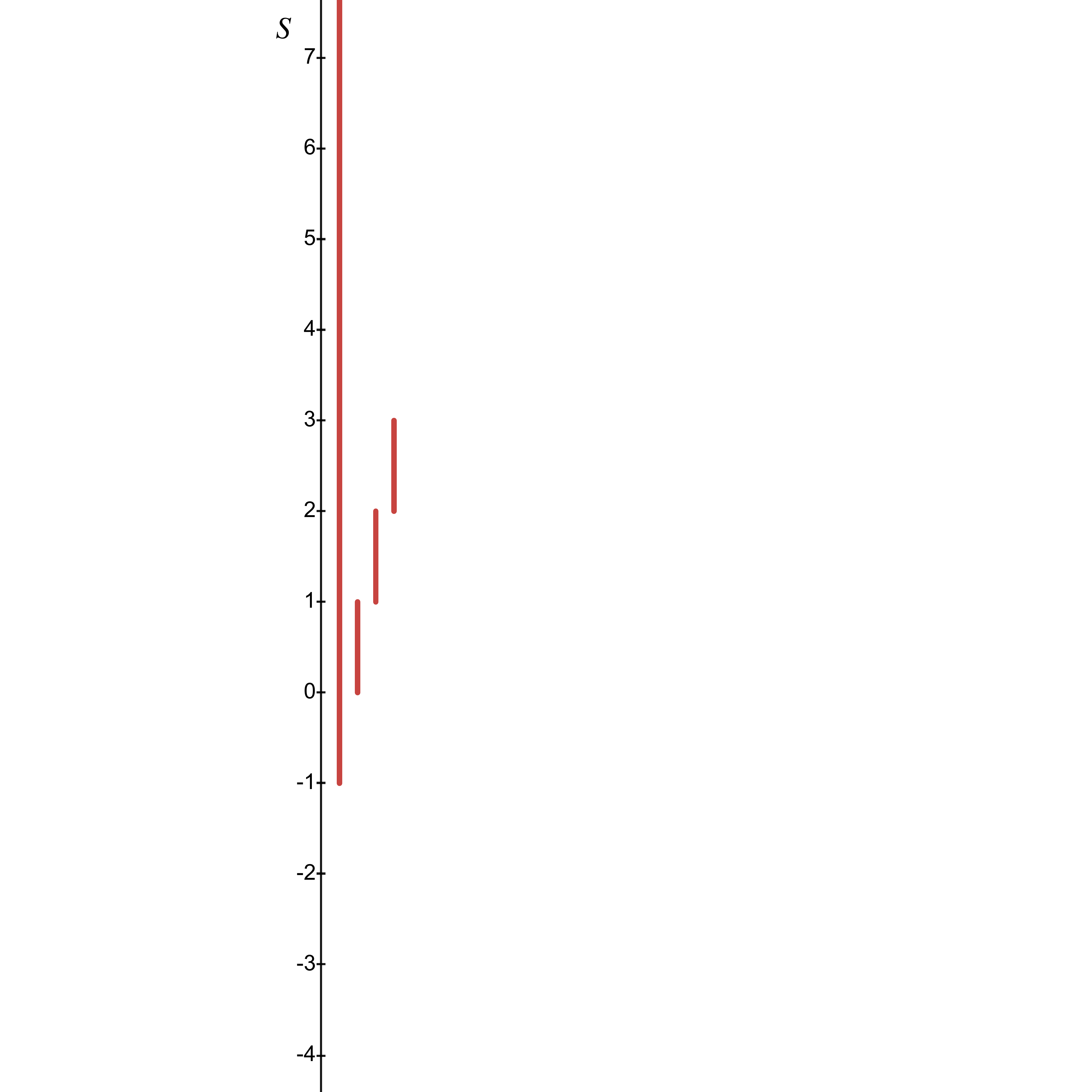}
	\end{minipage}
	\caption{Left: the sawtooth space $\Sigma_2 \subset \RR^2$. Right: the corresponding barcode for $\PH_0(\Sigma_2^g)$.}
	\label{fig:sawtooth2}
\end{figure}
Now, let $g: \RR^2 \to \RR$ denote the projection $(x,y) \mapsto y$ (and its restriction to $\Sigma_c$), and form the sublevel set filtration of the pair $(\Sigma_c, g)$. The degree $d=0$ persistent homology $\PH_0(\Sigma_c^g)$ is an $\RR$-persistence module.

\begin{thm}
$\PH_0(\Sigma_c^g)$ is barcode-finite and $\Qcodim(\PH_0(\Sigma_c^g)) = c$.
\end{thm}

\begin{proof}
For $s<-1$, the sublevel set $(\Sigma_c)_s$ is empty. At $s=-1$, a connected component is born which persists as $s\to \infty$. However, at each integer $m\in S$ with $0\leq m \leq c$, a new connected component is born (at the southern tip of a sawtooth), but which subsequently dies at $m+1 \in S$ (by coinciding with the persistent connected component born at $s=-1$). Hence, the resulting barcode consists of a sequence of bars $\langle -1,\infty \rangle ,\langle 0,1 \rangle ,\langle 1,2 \rangle ,\ldots,\langle c,c+1 \rangle $. The interacting pairs are exactly those of the form $(\langle k-1,k \rangle ,\langle k,k+1 \rangle )$ for integers $k$ in the range $1 \leq k \leq c$. The result follows from Definition \ref{defn:Qcodim}.
\end{proof}

\section{Discussion and future work}
    \label{s:Discuss.Future}

\subsection{Interpretation of quiver codimension as a statistic in TDA}
	\label{ss:interpret}
The calculation of $\Qcodim$ by counting pairs of interacting bars admits the following intuitive meaning in the setting of PH as a tool for \emph{topological data analysis} (TDA). 

Consider the topological features of the underlying data (measured by bars in the barcode) at many different possible ``scales" (measured by the birth and death times in the indexing set $S$). A positive codimension; i.e., \emph{interaction} between ``features" in the sense of Definition \ref{defn:interacting.pair}, means that the features persist at scales in $S$ which are \emph{different} (they necessarily must have distinct birth and death times) but they \emph{overlap}. Said another way, the larger the quiver codimension associated to the (PH of some filtration on the) data, the less likely we expect to be able to isolate the topological features of the data into separable scales.

\subsection{Distribution of quiver codimension}
Consider the case of a fixed metric space $(M,\rho)$ and a fixed positive integer $N$. Suppose we take finite \emph{random samples} of point clouds $X = \{x_i\}_{i=1}^N$ from $M$. Now, fix an integer $d\geq 0$ and suppose we consider $\Qcodim(\PH_d(\VR(X,M)))$ as a sample statistic. We can ask 
	\begin{question}
	How is the $\Qcodim$ statistic distributed? 
	\end{question}
Posing this question certainly depends on the structure of $M$ and a compatible formulation of ``random sample from $M$"; nonetheless, an answer to the question (or an analogous question for other filtrations) sheds light on the relationship between PT-stability and RT-stability. In particular, given the qualitative interpretations of the previous subsection, we might reasonably expect a unimodal distribution clustered just above the number of scales at which the major topological features interact in the population space $M$ (since most samples should capture the major features, and ``noise" can only add to the codimension but not subtract).  PT-stability ensures that the barcode should be stable relative to small perturbations in the data which we may expect from random sampling. RT-stability ensures that $\Qcodim$ is well-defined and computable on each sample, but to what extent does ``noise" in the barcode affect the variation of $\Qcodim$?

\subsection{``Higher" algebro-geometric invariants}

Let $\Omega \subset \Rep$ denote the isoclass of a given quiver representation. By definition, the closure $\bar{\Omega}$ is a $G_\alpha$-stable subvariety of the vector space $\Rep$ and furthermore defines an equivariant cohomology class $[\bar{\Omega}] \in H^*_{G_\alpha}(\Rep)$ \cite{B-F}. Since $\Rep$ is equivariantly contractible, $H^*_{G_\alpha}(\Rep)$ is isomorphic to a polynomial ring in the Chern classes of $G_\alpha$ and $[\bar{\Omega}]$ is called the \emph{quiver polynomial}. This polynomial invariant has been well-studied and exhibits remarkable combinatorial, geometric, stability, and positivity properties. One important property in the present context is that the quiver polynomial is homogeneous of degree equal to $\codim(\Omega)$; see e.g., \cite{B-F,BR,KMS,Rim}.

\begin{question}
    Does the quiver polynomial exhibit RT-stability or only its degree?
\end{question}

That is, it would be interesting to consider an extension of the quiver polynomial to general persistence modules which can be approximated by quiver representations in the sense of Definition \ref{defn:h.approx}. Furthermore, there is a plethora of characteristic classes associated to an isoclass $\Omega$; e.g., the Chern--Schwartz--Macpherson (CSM) and Segre--Schwartz--Macpherson (SSM) classes in cohomology as well as their K-theoretic, and elliptic-cohomology versions; see e.g., \cite{FRW} for a survey on the current state of the art. These ``higher'', non-homogeneous invariants typically recover the quiver polynomial as a specialization to their lowest degree term. The SSM classes may be of particular interest in data analysis, as they can satisfy a normalization property akin to a probability distribution; see, e.g., \cite[Remark 8.8]{FR}.

\section*{Acknowledgements}
We would like to acknowledge support from the Wake Forest University UReCA (\textbf{U}ndergraduate \textbf{Re}search and \textbf{C}reative \textbf{A}ctivities) Center for their award of a Wake Forest Research Fellowship to the second author during the summer of 2022 when this work was initiated. We would also like to thank Max Wakefield for providing us with introductory notes and thoughts on the subject of persistent homology, and Nicole Dalzell for suggestions, tips, and support with the statistical software package, R. Finally, we thank the anonymous referee for many useful comments and suggestions to improve the structure and exposition of this paper, particularly those regarding the general theory of persistence.

\bibliographystyle{plainurl}
\bibliography{BarcodeCodimBib}
\end{document}